\documentclass[12pt,amscd, amssymb]{amsart}
\usepackage{amsmath,amssymb,mathtools}
\usepackage{enumitem}
\usepackage{microtype}
\usepackage{hyperref}

\newtheorem{theorem}{Theorem}[section]
\newtheorem{proposition}[theorem]{Proposition}
\newtheorem{lemma}[theorem]{Lemma}
\newtheorem{corollary}[theorem]{Corollary}
\theoremstyle{remark}
\newtheorem{remark}[theorem]{Remark}

\newtheorem{example}[theorem]{Example}
\DeclareMathOperator{\ar}{ar}
\DeclareMathOperator{\gr}{gr}

\DeclareMathOperator{\rt}{rt}

\newcommand{\Jbar}{\overline J}
\newcommand{\Rbar}{\overline R}
\newcommand{\QNVT}{\mathrm{QNVT}}
\newcommand{\rev}{\mathrm{rev}}
\newcommand{\Rees}{\mathfrak{R}}

\def\reg{\operatorname{reg}}

\def\mm{{\frak m}}

\def\gr{{\mathrm{gr}}}
\def\ar{{\mathrm{ar}}}

\def\reg{\operatorname{reg}}

\def\dim{\operatorname{dim}}

\def\max{\operatorname{max}}
\newcommand{\inm}{\operatorname{in}}

\begin{document}

\title[A sharper perturbation bound]{An Improved Upper Bound for the Perturbation Index}
\author{Ton That Quoc Tan}
\address{Department of Mathematics, FPT University, Da Nang, Vietnam}
\email{tanttq@fe.edu.vn, quoctanmath@gmail.com}
\subjclass[2020]{Primary 13A30, 13D45; Secondary 13H15, 14B12}
\keywords{small perturbation, filter regular sequence, initial ideal, associated graded ring, Artin-Rees number}

\begin{abstract}
Let $(R,\mathfrak m)$ be a Noetherian local ring and $J$ be an arbitrary ideal of $R$. Suppose that $f_1,\ldots,f_r$ is a $J$-filter regular sequence in $R$ and $I=(f_1,\ldots,f_r)$. In this paper, we establish an improved upper bound for the perturbation index of the associated graded ring $\gr_J(R/I)$, refining the bound obtained by Quy and N. V. Trung.
\end{abstract}

\maketitle

\section{Introduction}

Let $(R,\mm)$ be a Noetherian local ring and let $I,J\subseteq R$ be ideals. Suppose that $I=(f_1,\ldots,f_r)$. An ideal
$I'=(f_1+\varepsilon_1,\ldots,f_r+\varepsilon_r)$, where $\varepsilon_i\in J^N$ for $i=1,\ldots,r$ and $N\gg0$, is said to be a \emph{$J$-adic perturbation} of $I$. If $J=\mm$, we simply call $I'$ a \emph{small perturbation} of $I$.
Set
$$
\gr_J(R/I)=\bigoplus_{n\ge0}\frac{J^n+I}{J^{n+1}+I},
$$
the associated graded ring of $R/I$ with respect to the ideal $(J+I)/I$. A natural perturbation problem asks when
$
\gr_J(R/I')\cong \gr_J(R/I)
$
for all sufficiently high $J$-adic perturbations $I'$ of $I$. The least integer \(N\) such that
$ \operatorname{gr}_J(R/I')\cong \operatorname{gr}_J(R/I) $
for every perturbation
$ I'=(f_1+\varepsilon_1,\ldots,f_r+\varepsilon_r) \quad\text{with }\varepsilon_i\in J^N $
is called the \emph{perturbation index} of \(\operatorname{gr}_J(R/I)\).

When $J$ is $\mm$-primary, this problem is closely related to the invariance of Hilbert-Samuel functions under small perturbations. The subject goes back to Eisenbud's work on adic approximation of complexes \cite{Eis}. Srinivas and Trivedi \cite{ST} studied the stability of Hilbert-Samuel functions and formulated a conjecture in the filter-regular case, which was later solved by Ma, Quy and Smirnov \cite{MQS}. For generalized Cohen-Macaulay local rings, Quy and V. D. Trung \cite{QVDT} obtained a linear upper bound in terms of the homological degree for the perturbation index preserving the Hilbert function. In a dual direction, Tan \cite{TanBKMS} proved the stability of the Hilbert--Samuel function relative to an Artinian module under perturbations generated by a coregular sequence. On the homological side, Duarte \cite{Duarte} studied Betti numbers under small perturbations, and more recently V. D. Trung \cite{VDT} established the stability of both Bass and Betti numbers under ideal perturbations.

For an arbitrary ideal $J$ of $R$, Quy and N. V. Trung \cite{QT} proved that if $f_1,\ldots,f_r$ is a $J$-filter regular sequence then sufficiently small $J$-adic perturbations preserve the initial ideal $\inm(I)$ and therefore preserve $\gr_J(R/I)$.  They obtained the explicit bound
\begin{equation}\label{eq:intro-QT-bound}
 N_{\QNVT}
 =\max\left\{
 a_1+2a_2+4a_3+\cdots+2^{r-1}a_r,
 A_1,\ldots,A_r
 \right\}+1,
\end{equation}
where
$$
 a_i=
 a_J\!\left(
 \frac{(f_1,\ldots,f_{i-1}):f_i}{(f_1,\ldots,f_{i-1})}
 \right) \quad \text{and} \quad  A_i=\ar_J(f_1,\ldots,f_i)
$$
for $i=1\ldots r$. 


The purpose of this paper is to show that one full doubling can be avoided.  The main result is the following.

\begin{theorem}[Main theorem]\label{thm:intro-main}
Let $J$ be an arbitrary ideal of a Noetherian local ring $(R,\mm)$ and let
$f_1,\ldots,f_r$ be a $J$-filter regular sequence.  For $i=1,\ldots,r$, let $$a_i =  a_J\!\left(
 \frac{(f_1,\ldots,f_{i-1}):f_i}{(f_1,\ldots,f_{i-1})}
 \right) \quad \text{and} \quad A_i=\ar_J(f_1,\ldots,f_i). $$ 
 Set
\begin{equation}\label{eq:intro-new-bound}
 N
 =\max\left\{
 a_1+a_2+2a_3+4a_4+\cdots+2^{r-2}a_r,
 A_1,\ldots,A_r
 \right\}+1.
\end{equation}
For all $\varepsilon_1,\ldots,\varepsilon_r\in J^{N}$ and $f_i'=f_i+\varepsilon_i$ for $i=1,\ldots,r$. Then
\begin{enumerate}[label=\textup{(\roman*)}]
\item $f_1',\ldots,f_r'$ is a $J$-filter regular sequence with  $$a_J\left(\dfrac{(f_1',\ldots,f_{i-1}'):f_i')}{(f_1',\ldots,f_{i-1}')}\right) \leq 2^{i-1}a_i$$ for $i=1,\ldots,r$.
\item  For  $j=1,\ldots,r$, we have
$
 \inm(f_1',\ldots,f_j'))=\inm(f_1,\ldots,f_j).$

\item In particular, 
$
 \gr_J(R/I')\cong\gr_J(R/I)$ and  $\ar_J(I')=\ar_J(I).
$
\end{enumerate}
\end{theorem}
The idea is elementary but order-sensitive.  We first perturb
\[
 f_r,f_{r-1},\ldots,f_2
\]
in this order while keeping $f_1$ fixed.  At the stage where $f_k$ is perturbed, the prefix $f_1,\ldots,f_{k-1}$ is still unchanged, so

$$ a_J\!\left(\frac{(f_1,\ldots,f_{k-1}):f_k}{(f_1,\ldots,f_{k-1})}\right)=a_k. $$
 This yields the coefficient pattern
\[
 1,1,2,4,\ldots,2^{r-2}
\]
in the sufficient perturbation bound.

If we write
\[
 B_{\QNVT}=a_1+2a_2+\cdots+2^{r-1}a_r,
 \qquad
 B_{\rev}=a_1+a_2+2a_3+\cdots+2^{r-2}a_r,
\]
then
\begin{equation}\label{eq:exact-comparison}
 B_{\QNVT}=2B_{\rev}-a_1.
\end{equation}
Thus the new weighted term is exactly $(B_{\QNVT}+a_1)/2$.  The improvement may of course be hidden if one of the Artin--Rees numbers $A_i$ dominates both weighted terms.  When $a_1=\cdots=a_r=1$ and the Artin--Rees terms vanish, the two sufficient exponents are respectively
\[
 2^r \qquad\text{and}\qquad 2^{r-1}+1.
\]
We construct such an example in \ref{example}.

The paper is organized as follows.  Section~2 recalls the basic facts from \cite{QT}.  Section~3 performs the reverse-tail induction and proves Theorem~\ref{thm:intro-main}.  Section~\ref{sec:applications} discusses several consequences and applications of the improved perturbation bound.
\section{Preliminaries}\label{sec:prelim}

Throughout this paper $(R,\mm)$ is a Noetherian local ring and $J\subseteq R$ is an arbitrary ideal.  For a finitely generated $R$-module $M$, set
\[
 a_J(M)=\inf\{n\ge0\mid J^nM=0\},
\]
with the convention $a_J(M)=\infty$ if no such $n$ exists. We call $a_J(M)$ the \emph{$J$-Loewy length} of $M$. 

A sequence $f_1,\ldots,f_r$ in $R$ is called \emph{$J$-filter regular} if for every $i =1, \ldots, r$,
$f_i$ avoids every associated prime of $R/(f_1,\ldots,f_{i-1})$ which does not contain $J$.  Equivalently, for $i=1, \ldots, r$
\[
 a_J\!\left(
 \frac{(f_1,\ldots,f_{i-1}):f_i}{(f_1,\ldots,f_{i-1})}
 \right)<\infty.
\]

For an ideal $I\subseteq R$, let $\inm(I)$ denote its \emph{initial ideal} of $I$ in $\gr_J(R)$.  We use the convention of \cite{QT}; in particular,
\[
 \gr_J(R/I)\cong \gr_J(R)/\inm(I).
\]
The \emph{Artin--Rees number of $I$ with respect to $J$}, denoted $\ar_J(I)$, is the least integer $c$ such that
\[
 J^{n+1}\cap I=J(J^n\cap I)
 \qquad\text{for all }n\ge c.
\]
By \cite[Proposition 2.3]{QT},
\begin{equation}\label{eq:AR-initial}
 \ar_J(I)=d(\inm(I)),
\end{equation}
where $d(\inm(I))$ is the maximal degree of a minimal homogeneous generating set of $\inm(I)$.

We collect the precise results from \cite{QT} used below.

\begin{lemma}\label{lem:one}
Let $(R,\mm)$ be a Noetherian local ring and let $f\in R$ be a $J$-filter regular element.  Set
\[
 c=\max\{a_J(0:f),\ar_J(f)+1\}.
\]
If $f'=f+\varepsilon$ with $\varepsilon\in J^c$, then
\[
 0:f'=0:f,
 \qquad
 \inm(f')=\inm(f).
\]
Consequently $f'$ is $J$-filter regular and
$a_J(0:f')=a_J(0:f)$.
\end{lemma}

\begin{proof}
This is \cite[Proposition 3.2]{QT}.
\end{proof}

\begin{lemma}\label{lem:switch}
For $u,v\in R$ there is an isomorphism
\[
 \frac{(u):v}{(u)+(0:v)}
 \cong
 \frac{(v):u}{(v)+(0:u)}.
\]
Equivalently, after quotienting by a common ideal $K$ of $R$,
\begin{equation}\label{eq:switchK}
 \frac{(K,u):v}{(K,u)+(K:v)}
 \cong
 \frac{(K,v):u}{(K,v)+(K:u)}.
\end{equation}
\end{lemma}

\begin{proof}
The first statement is \cite[Lemma 3.3]{QT}.  The second is the same result applied in $R/K$.
\end{proof}

\begin{lemma}\label{lem:ARquot}
Let $K\subseteq I$ and put $\Rbar=R/K$ and $\Jbar=(J+K)/K$.  Then
\[
 \ar_{\Jbar}(I/K)\le \ar_J(I).
\]
\end{lemma}

\begin{proof}
This is \cite[Lemma 2.6]{QT}.
\end{proof}

\begin{lemma}\label{lem:ARin}
If $K\subseteq I$, then
\[
 \inm(I/K)=\inm(I)/\inm(K)
\]
in the associated graded ring of $R/K$.  Moreover, if $\inm(I)=\inm(I')$, then
\[
 \ar_J(I)=\ar_J(I').
\]
\end{lemma}

\begin{proof}
The first statement is \cite[Lemma 2.2]{QT}; the second follows from \eqref{eq:AR-initial}.
\end{proof}

The key quantitative input is \cite[Proposition 3.4]{QT}.

\begin{proposition}\label{prop:pair}
Let $f_1,f_2$ be a $J$-filter regular sequence and put
\[
 a_1=a_J(0:f_1),
 \qquad
 a_2=a_J\!\left(\frac{(f_1):f_2}{(f_1)}\right).
\]
Set
\[
 c=\max\{a_1+a_2,\ar_J(f_1),\ar_J(f_1,f_2)\}+1.
\]
If $f_1'=f_1+\varepsilon_1$ with $\varepsilon\in J^c$, then $f_1',f_2$ is a $J$-filter regular sequence,
\[
 a_J(0:f_1')=a_1,
 \qquad
 a_J\!\left(\frac{(f_1'):f_2}{(f_1')}\right)\le 2a_2,
\]
and
\[
 \inm(f_1',f_2)=\inm(f_1,f_2).
\]
\end{proposition}

\section{Main Results}\label{sec:reverse}
In this section, we will prove the Theorem \ref{thm:intro-main}.

The following lemma is a new ingredient in our reverse perturbation argument, inspired by the proof of \cite[Theorem 3.5]{QT}.
\begin{lemma}\label{lem:first-entry}
Let $
 f_1,\ldots,f_r
$
be a $J$-filter regular sequence. For $ i=1,\ldots,r$,  put
\[
 a_i=
 a_J\!\left(
 \frac{(f_1,\ldots,f_{i-1}):f_i}{(f_1,\ldots,f_{i-1})}
 \right).
\]
Set
\[
 L=\max\left\{
 a_1+\cdots+a_r,
 \ar_J(f_1),\ldots,\ar_J(f_1,\ldots,f_r)
 \right\}+1.
\]
If
$
 f_1'=f_1+\varepsilon$ where  $\varepsilon \in J^L,$ then
\begin{enumerate}[label=\textup{(\roman*)}]
 \item $f_1',f_2,\ldots,f_r$ is a $J$-filter regular sequence.
 \item $a_J(0:f_1')=a_1$ and for every $i=2,\ldots,r$,
$$
 a_J\!\left(
 \frac{(f_1',f_2,\ldots,f_{i-1}):f_i}
 {(f_1',f_2,\ldots,f_{i-1})}
 \right)
 \le 2a_i.
 $$
 \item For every $i=1,\ldots,r$,
$$
 \inm(f_1',f_2,\ldots,f_i)
 =
 \inm(f_1,f_2,\ldots,f_i).
$$
\end{enumerate}
\end{lemma}

\begin{proof}
For $1 \le i\le r$, put
$
 K_i=(f_2,\ldots,f_i),
$
with $K_1=(0)$.  We first claim that
\begin{equation}\label{eq:reverse-colon-bound}
 a_J\!\left(\frac{K_i:f_1}{K_i}\right)
 \le a_1+\cdots+a_i
 \qquad(i=1,\ldots,r).
\end{equation}
For $i=1$ this is the definition of $c_1$.

Assume $i\ge2$ and that \eqref{eq:reverse-colon-bound} is known for $i-1$.  Apply Lemma~\ref{lem:switch} in 
$\overline{R} = R/(f_2,\ldots,f_{i-1})$ to the images of $f_1$ and $f_i$.  We obtain
\begin{equation}\label{eq:valid-switch-step}
 \frac{(f_1,\ldots,f_{i-1}):f_i}
 {(f_1,\ldots,f_{i-1})+(f_2,\ldots,f_{i-1}):f_i}
 \cong
 \frac{K_i:f_1}{K_i+(K_{i-1}:f_1)}.
\end{equation}
Since
\[
 J^{a_i}\big((f_1,\ldots,f_{i-1}):f_i\big)
 \subseteq(f_1,\ldots,f_{i-1}),
\]
it follows that
\[
 J^{a_i}\big((f_1,\ldots,f_{i-1}):f_i\big)
 \subseteq(f_1,\ldots,f_{i-1})+(f_2,\ldots,f_{i-1}):f_i.
\]
Therefore, the left-hand quotient in \eqref{eq:valid-switch-step} is annihilated by $J^{a_i}$.  Hence
\begin{equation}\label{eq:one-colon-step}
 J^{a_i}(K_i:f_1)
 \subseteq K_i+(K_{i-1}:f_1).
\end{equation}
By the induction hypothesis,
\[
 J^{a_1+\cdots+a_{i-1}}(K_{i-1}:f_1)\subseteq K_{i-1}\subseteq K_i.
\]
Multiplying \eqref{eq:one-colon-step} by $J^{a_1+\cdots+a_{i-1}}$, we have
$$J^{a_1+\cdots+a_{i-1}+a_i}(K_i:f_1) \subseteq J^{a_1+\cdots+a_{i-1}}(K_i+(K_{i-1}:f_1)) \subseteq K_i.  $$
It proves \eqref{eq:reverse-colon-bound}.

Now fix $i\in\{1,\ldots,r-1\}$ and set 
$
 R_i=R/K_i=R/(f_2,\ldots,f_i).
$
By \eqref{eq:reverse-colon-bound}, the image of $f_1$ is a $J$-filter regular element in $R_i$ and
\[
 a_J(0_{R_i}:f_1)\le a_1+\cdots+a_i < \infty.
\]
On the other hand,
$$a_J\left(\frac{f_1R_i:f_{i+1}}{f_1R_i}\right)= a_J\left(\frac{(f_1,f_2,\ldots,f_{i}):f_{i+1}}{(f_1,f_2,\ldots,f_{i})}\right)=a_{i+1}<\infty.
$$
  Thus $f_1,f_{i+1}$ is a $J$-filter regular sequence in $R_i$.

By Lemma~\ref{lem:ARquot},
\begin{align*}
 \ar_{JR_i}(f_1R_i)&\le \ar_J(f_1,\ldots,f_i),\\
 \ar_{JR_i}((f_1,f_{i+1})R_i)&\le \ar_J(f_1,\ldots,f_{i+1}).
\end{align*}
Therefore,
$$L \geq \max\left\{a_J(0_{R_i}:f_1) + a_J\left(\frac{f_1R_i:f_{i+1}}{f_1R_i}\right), \ar_{JR_i}(f_1R_i), \ar_{JR_i}((f_1,f_{i+1})R_i) \right\}+1$$
This allows Proposition~\ref{prop:pair} to be applied in $R_i$.  We obtain
\[
 a_J\left(\frac{f_1'R_i:f_{i+1}}{f_1'R_i}\right) \leq 2a_J\left(\frac{f_1R_i:f_{i+1}}{f_1R_i}\right) 
 = 2a_{i+1}
\]
It follows that 
$$a_J\!\left(
 \frac{(f_1',f_2,\ldots,f_i):f_{i+1}}
 {(f_1',f_2,\ldots,f_i)}
 \right) \leq  2a_{i+1} $$
In addition, 
$$\inm((f_1',f_{i+1})R_i) = \inm((f_1,f_{i+1})R_i).$$
It follows that $f_1', f_2,\ldots,f_r$ is a $J$-filter regular sequence in $R$ and 
\[
 \inm(f_1',f_2,\ldots,f_{i+1})
 =
 \inm(f_1,f_2,\ldots,f_{i+1}).
\]
Finally, Lemma~\ref{lem:one} applied to $f_1$ gives
$a_J(0:f_1')=a_1$ and $\inm(f_1')=\inm(f_1)$.  As the preceding argument works for every $i=1,\ldots,r-1$, the proof is complete.
\end{proof}
For $i=1,\ldots,r$, set
$$
 a_i=a_J\!\left(\frac{(f_1,\ldots,f_{i-1}):f_i}{(f_1,\ldots,f_{i-1})}\right).
$$
The reverse-order construction first perturbs
\[
 f_r,f_{r-1},\ldots,f_2
\]
while keeping $f_1$ fixed.  For $2\le k\le r$, after the entries $f_r,\ldots,f_k$ have been perturbed, write
\[
 \mathbf f^{[k]}
 =
 f_1,\ldots,f_{k-1},f_k',f_{k+1}',\ldots,f_r'.
\]
For $j\ge k$, set
$$
 d_j^{[k]}
 =a_J\!\left(
 \frac{(f_1,\ldots,f_{k-1},f_k',\ldots,f_{j-1}'):f_j'}
 {(f_1,\ldots,f_{k-1},f_k',\ldots,f_{j-1}')}
 \right).
$$

\begin{proposition}\label{prop:reverse-tail}
Assume
\[
 N = \max\{a_1+a_2+2a_3+\cdots+2^{r-2}a_r,\ar_J(f_1),\ldots,\ar(f_1,\ldots,f_r)\}+1
\]
and $f_i'=f_i+\varepsilon_i$ with $\varepsilon_i\in J^N$.  Then, after perturbing $f_r,f_{r-1},\ldots,f_k$ in that order, for every $k=2,\ldots,r$ the following hold.
\begin{enumerate}[label=\textup{(\roman*)}]
 \item $\mathbf f^{[k]} = f_1,\ldots,f_{k-1},f_k',f_{k+1}',\ldots,f_r'$ is a $J$-filter regular sequence with  
$$ d_j^{[k]}\le 2^{j-k}a_j;$$
 for $j\ge k$. In particular $d_k^{[k]}=a_k$.
 \item For every $i=1,\ldots,r$,
$$
 \inm(f_1,\ldots,f_{k-1},f_k',\ldots,f_i')
 =\inm(f_1,\ldots,f_{k-1},f_k,\ldots,f_i).
$$
\end{enumerate}
\end{proposition}

\begin{proof}
We argue by descending induction on $k$.

For the base case $k=r$, set $\overline{R}=R/(f_1,\ldots,f_{r-1})$.  The image of $f_r$ is a $J$-filter regular element in $\overline{R}$ with 
$$a_{J}(0_{\overline{R}}:f_r) = a_J\left(\frac{(f_1,\ldots,f_{r-1}):f_r}{(f_1,\ldots,f_{r-1})}\right)=a_r.$$
and Lemma~\ref{lem:ARquot} gives
\[
 \ar_{J\overline{R}}(f_r\overline{R})\le \ar_J(f_1,\ldots,f_r).
\]
Since $a_1+a_2+2a_3+\cdots+2^{r-2}a_r\ge a_r$ and $$N\ge\max\{a_r,\ar(f_1,\ldots,f_r)\}+1.$$ By Lemma~\ref{lem:one}, we get $f_r'$ is a $J$-filter regular element in $\overline{R}$ and
$$a_{J}(0_{\overline{R}}:f_r') =a_{J}(0_{\overline{R}}:f_r)=a_r.$$
On the other hand,
 $$d_r^{[r]} = a_J\left(\frac{(f_1,\ldots,f_{r-1}):f_r'}{(f_1,\ldots,f_{r-1})}\right) = a_r.$$
and
\[
 \inm(f_r'\overline{R})=\inm(f_r\overline{R}).
\]
Thus 
$$\inm(f_1,\ldots,f_{r-1},f_r') = \inm(f_1,\ldots,f_{r-1},f_r).$$
Thus, the claim holds for $k=r$.\\
Now let $1\leq k<r$ and set
\[
S=R/(f_1,\ldots,f_{k-1}).
\]
By the descending induction hypothesis, the assertion holds for $k+1$. We may assume that
\begin{enumerate}[label=\textup{(\roman*')}]
\item $\mathbf f^{[k+1]}=f_1,\ldots,f_k,f_{k+1}',\ldots,f_r'$ is a $J$-filter regular sequence in $R$ with
$$d_{k+1}^{[k+1]} = a_J(f_kS:f_{k+1}) = a_J\left(\frac{(f_1,\ldots,f_{k}):f_{k+1}}{(f_1,\ldots,f_{k})}\right)=a_{k+1}$$
and 
 $$d_j^{[k+1]}  =a_J\!\left(
 \frac{(f_1,\ldots,f_{k},f_{k+1}',\ldots,f_{j-1}'):f_j'}
 {(f_1,\ldots,f_{k},f_{k+1}',\ldots,f_{j-1}')}
 \right) \leq  2^{j-k-1}a_j$$
 for $j=k+2,\ldots,r$.
 
 \item $\inm((f_k,f_{k+1}',\ldots,f_i')S) = \inm((f_k,f_{k+1},\ldots,f_i)S)$ for $i=k,\ldots,r$.
\end{enumerate}
We will prove the assertion holds for $k.$
For every $k\ge2$, by (i')
\begin{align*}\label{1}
  a_J\left(0_S:f_k\right)+a_J\left(\frac{f_kS:f_{k+1}'}{f_kS}\right)+& \cdots +  a_J\left(\frac{(f_k,f_{k+1}',\ldots,f_{r-1}')S:f_{r}'}{(f_k,f_{k+1}',\ldots,f_{r-1}')S}\right)   \\
   & \leq a_k+a_{k+1}+2a_{k+2}+\cdots+2^{r-k-1}a_r= C_k.
\end{align*}
By Lemma ~\ref{lem:ARin} and (ii') implies
$$\inm(f_1,\ldots,f_k,f_{k+1}',\ldots,f_j') = \inm(f_1,\ldots,f_k,f_{k+1},\ldots,f_j)$$
and
$$\ar_J(f_1,\ldots,f_k,f_{k+1}',\ldots,f_j') = \ar_J(f_1,\ldots,f_k,f_{k+1},\ldots,f_j)$$
for $j\ge k$.

On the ring $S=R/(f_1,\ldots,f_{k-1})$, by Lemma~\ref{lem:ARquot}, we have
$$\ar_{JS}(f_kS) \leq \ar_J(f_1,\ldots,f_{k-1},f_k)$$
and
$$\ar_{JS}((f_k,f_{k+1}'\ldots,f_i')S) \leq \ar_J(f_1,\ldots,f_k,f_{k+1}',\ldots,f_i')$$
for $i \geq k$.
It follows that
$$\max\{C_k,\ar_{JS}(f_kS),\ldots, \ar_{JS}((f_k,f_{k+1}'\ldots,f_r')S) \}+1 \leq N$$
Applying Lemma~\ref{lem:first-entry} in $S$, we get
$f_k',f_{k+1}',\ldots,f_r'$ is a $J$-filter regular sequence in $S$ with $d_k^{[k]}=a_J(0_S:f_k')=a_k$ and for $j>k$
$$d_j^{[k]} \le 2 d_j^{[k+1]} \leq 2\cdot2^{j-k-1}a_j=2^{j-k}a_j.$$
On the other hand, for $i \geq k$
$$\inm((f_k',f_{k+1}',\ldots, f_{i}')S) = \inm((f_k,f_{k+1}',\ldots, f_{i}')S)$$
By (ii'), we have
$$\inm((f_k,f_{k+1}',\ldots, f_{i}')S) = \inm((f_k,f_{k+1},\ldots,f_i)S)$$
for $i \geq k$.
Hence
$$\inm((f_1,\ldots,f_{k-1},f_k',\ldots, f_i'))=\inm((f_1,\ldots,f_{k-1},f_k,\ldots, f_i))$$
for $i \geq k$.
This completes the descending induction.
\end{proof}

\begin{theorem}\label{thm:main}
Let $J$ be an arbitrary ideal of a Noetherian local ring $(R,\mm)$ and let
$
 I=(f_1,\ldots,f_r),
$
where $f_1,\ldots,f_r$ is a $J$-filter regular sequence and $r\ge2$.  Define $$
 a_i=a_J\!\left(\frac{(f_1,\ldots,f_{i-1}):f_i}{(f_1,\ldots,f_{i-1})}\right).
$$ Set
$$
 N=
 \max\left\{
 a_1+a_2+2a_3+\cdots+2^{r-2}a_r,
 \ar_J(f_1),\ldots,\ar(f_1,\ldots,f_r)
 \right\}+1.
$$
Let
\[
 f_i'=f_i+\varepsilon_i,
 \qquad \varepsilon_i\in J^N,
 \qquad I'=(f_1',\ldots,f_r').
\]
Then
\begin{enumerate}[label=\textup{(\roman*)}]
 \item $f_1',\ldots,f_r'$ is a $J$-filter regular sequence.
 \item For every $i=1,\ldots,r$,
$$
 a_J\!\left(
 \frac{(f_1',\ldots,f_{i-1}'):f_i'}
 {(f_1',\ldots,f_{i-1}')}
 \right)
 \le2^{i-1}a_i.
$$
 \item For every $i=1,\ldots,r$,
$$
 \inm(f_1',\ldots,f_i')=\inm(f_1,\ldots,f_i);
$$
 in particular $\inm(I')=\inm(I)$.
 \item 
 \[
 \gr_J(R/I')\cong\gr_J(R/I),
 \qquad
 \ar_J(I')=\ar_J(I).
 \]
\end{enumerate}
\end{theorem}

\begin{proof}
Apply Proposition~\ref{prop:reverse-tail} down to $k=2$.  We obtain the sequence
$$
 f_1,f_2',\ldots,f_r',
$$
which is $J$-filter regular and
$$
 c_1=a_J(0:f_1)=a_1,
 \qquad c_2=a_J((f_1):f_2'/(f_1))=a_2,
 $$
 $$
 \qquad c_i =  a_J\!\left(
 \frac{(f_1,f_2'\ldots,f_{i-1}'):f_i'}
 {(f_1,f_2'\ldots,f_{i-1}')}
 \right)\le2^{i-2}a_i\quad(i\ge3).
$$
Consequently
$$
 c_1+\cdots+c_r
 \le a_1+a_2+2a_3+\cdots+2^{r-2}a_r.
$$
Moreover Proposition~\ref{prop:reverse-tail}(iii) gives
\[
 \inm(f_1,f_2',\ldots,f_i')=\inm(f_1,\ldots,f_i).
\]
for $i=1,\ldots,r$.
Thus Lemma~\ref{lem:ARin} implies
$$
 \ar_J(f_1,f_2',\ldots,f_i')=\ar_J(f_1,\ldots,f_i)
$$ for $i=1,\ldots,r$.
Therefore,
$$N \geq \max\{c_1+\cdots+c_r,\ar_J(f_1),\ldots,\ar_J(f_1,\ldots,f_r)\}+1$$
We now apply Lemma~\ref{lem:first-entry}, we have $f_1',f_2',\ldots, f_r'$ is $J$-filter regular sequence with
$$a_J(0:f_1') = a_J(0:f_1)$$
and for $i=2,\ldots, r$
$$ a_J\!\left(
 \frac{(f_1',\ldots,f_{i-1}'):f_i'}
 {(f_1',\ldots,f_{i-1}')}
 \right) \leq 2  a_J\!\left(
 \frac{(f_1,f_2'\ldots,f_{i-1}'):f_i'}
 {(f_1,f_2'\ldots,f_{i-1}')}
 \right)
 \le2\cdot2^{i-2}a_i=2^{i-1}a_i;$$
 and for $i=1,\ldots,r$
 $$\inm(f_1',f_2'\ldots,f_i) = \inm(f_1,\ldots,f_i).$$
This proves (i)--(iii).  Finally (iv) follows from $\inm(I')=\inm(I)$, the identification
\[
 \gr_J(R/I)\cong\gr_J(R)/\inm(I),
\]
and \eqref{eq:AR-initial}.
\end{proof}

\begin{remark}
  When $r=2$, we get $$N= \max\{a_1+a_2, \ar_J(f_1),\ar_J(f_1,f_2)\} +1.$$
\end{remark}

\begin{example}\label{example}

Let $k$ be a field and 
$$
 A=k[[x_1,\ldots,x_r,y]],
 \qquad
 \mathfrak q=(x_1,\ldots,x_r,y).
$$
Regard $k=A/\mathfrak q$ as an $A$-module and form the idealization
\[
 R=A\ltimes k=A\oplus k,
\]
with multiplication
\[
 (a,\lambda)(b,\mu)=(ab,\overline a\mu+\overline b\lambda),
\]
where bars denote residue classes in $k=A/\mathfrak q$.  Then $R$ is a Noetherian local ring with maximal ideal $\mathfrak q\ltimes k$.

Set
\[
 J=((y,0))R
 \qquad\text{and}\qquad
 f_i=(x_i,0),\quad i=1,\ldots,r.
\]
For $I_i=(x_1,\ldots,x_i)A$ and $F_i=(f_1,\ldots,f_i)R$, a direct computation gives
\[
 F_i=I_i\oplus0,
 \qquad
 a_i=a_J\left(\frac{F_{i-1}:f_i}{F_{i-1}}\right)=1,
 \qquad
 \ar_J(F_i)=0
 \qquad(i=1,\ldots,r).
\]
Consequently the Quy-N. V. Trung bound gives
\[
 N_{\QNVT}=2^r,
\]
whereas Theorem \ref{thm:main} gives
\[
 N=2^{r-1}+1.
\]
In particular, for $r=4$ the two sufficient exponents are respectively $16$ and $9$.
\end{example}

\section{Consequences and applications}\label{sec:applications}
Throughout this section $N$ is as in Theorem \ref{thm:main} and $I'$ is any perturbation of order $N$.

\begin{corollary}[Hilbert--Samuel functions]
If $J$ is $\mm$-primary then
\[
 \ell\bigl(R/(I+J^n)\bigr)=\ell\bigl(R/(I'+J^n)\bigr)
 \qquad(n\ge0).
\]
\end{corollary}
\begin{proof}
This follows from $\gr_J(R/I)\cong\gr_J(R/I')$.
\end{proof}

\begin{corollary}[Achilles--Manaresi function; cf.\ {\cite[Cor. 3.11]{QT}}]\label{cor:AM}
$R/I$ and $R/I'$ have the same Achilles--Manaresi bivariate function with respect
to $J$, hence the same multiplicity sequence $c_0(J),\dots,c_d(J)$. In
particular, when $J$ is the Jacobian ideal, the Segre numbers are unchanged.
\end{corollary}

\begin{proof}
$\gr_\mm\big(\gr_J(R/I)\big)\cong\gr_\mm\big(\gr_J(R/I')\big)$.
\end{proof}

\begin{corollary}[Rees algebra; cf.\ {\cite[Cor. 3.12]{QT}}]\label{cor:rees}
The following invariants and properties of $\Rees_J(R/I)$ are shared by
$\Rees_J(R/I')$: the relation type $\rt$, the Castelnuovo--Mumford regularity
$\reg$, Cohen--Macaulayness, and Gorensteinness. Moreover $\ar_J(I')=\ar_J(I)$.
\end{corollary}

\begin{remark}[Perturbation index and extended degree]\label{rem:extdeg}
In \cite[Thm. 4.5]{QT} and \cite[Cor. 4.8]{QT} the perturbation index of
$\gr_J(R/I)$ is bounded in terms of an extended degree $D=D(J,R/I)$. The
numerical constraint used there,
$\sum_{i=1}^{r}2^{\,i-1}a_i\le(2^{r}-1)D\le N$, may now be replaced by
$a_1+a_2+2a_3+\cdots+2^{r-2}a_r \le2^{\,r-1}D\le N$. All the statements of \cite[Section 4]{QT} therefore hold
with this smaller first constraint; we do not restate the resulting formulas,
since in \cite[Thm. 4.5]{QT} the Artin--Rees term dominates.
\end{remark}

\subsection{Homological invariants}
 If $f$ is $J$-filter regular in $R$ and
$c\ge\max\{a_J(0:f),\ \ar_J(f)+1\}+1$, then for $\varepsilon\in J^{c}$ and
$f'=f+\varepsilon$ one has $$\beta^R_j(R/(f))=\beta^R_j(R/(f'))$$ and
$$\mu^j_R(R/(f))=\mu^j_R(R/(f'))$$ for all $j$ \cite[Prop. 3.4]{VDT}.

The our result also combines with V. D. Trung's one-element stability theorem for Bass and Betti numbers.  We use the result \cite{VDT} and combine with our reverse-order bound, it shows that if $\varepsilon_i\in J^{N^\sharp}$ for all $i$, where
$$
N^\sharp=\max\{a_1+a_2+2a_3+\cdots+2^{r-2}a_r,,
\operatorname{ar}_J(f_1),\ldots,\operatorname{ar}_J(f_1,\ldots,f_r)\}+2,
$$
then all Bass and Betti numbers of $R/I$ and $R/I'$ agree.

\begin{theorem}\label{thm:homological-app}

If
\[
 f_i'=f_i+\varepsilon_i,
 \qquad \varepsilon_i\in J^{N^\sharp},
\]
then, for every $j\ge0$,
\[
 \beta_j^R(R/I)=\beta_j^R(R/I'),
 \qquad
 \mu_R^j(R/I)=\mu_R^j(R/I').
\]
\end{theorem}

\begin{proof}
  For $k=0,\ldots,r$, set
\[
 I_k=(f_1,\ldots,f_k,f'_{k+1},\ldots,f'_r),
\]
so that $I_r=I$ and $I_0=I'$.  For $k=r,\ldots,1$, put
\[
 K_k=(f_1,\ldots,f_{k-1},f'_{k+1},\ldots,f'_r)
 \quad\text{and}\quad
 M_k=R/K_k.
\]
Then
\[
 R/I_k=M_k/f_kM_k,
 \qquad
 R/I_{k-1}=M_k/f'_kM_k.
\]
Let
\[
S_k=R/(f_1,\ldots,f_{k-1})
\]
and consider in $S_k$ the $J$-filter regular sequence
\[
f_k,f'_{k+1},\ldots,f'_r.
\]
By Proposition \ref{prop:reverse-tail}, its successive $J$-Loewy lengths are bounded by
\[
a_k,\quad a_{k+1},\quad
2a_{k+2},\quad \ldots,\quad
2^{r-k-1}a_r.
\]
Set
\[
L_k=(f'_{k+1},\ldots,f'_r)S_k.
\]
Since
\[
M_k=S_k/L_k,
\]
we have
\[
0:_{M_k}f_k
\cong
\frac{L_k:_{S_k}f_k}{L_k}.
\]
Applying estimate (6) of Lemma \ref{lem:first-entry} to the sequence
$f_k,f'_{k+1},\ldots,f'_r$ in $S_k$, we obtain
\[
a_J(0_{M_k}:f_k)
=
a_J\left(\frac{L_k:_{S_k}f_k}{L_k}\right)
\le
a_k+a_{k+1}+2a_{k+2}
+\cdots+2^{r-k-1}a_r.
\]
Hence
\[
a_J(0_{M_k}:f_k)\le a_1+a_2+2a_3+\cdots+2^{r-2}a_r.
\]
Moreover, Proposition \ref{prop:reverse-tail} gives
\[
 \operatorname{in}(I_k)=\operatorname{in}(I),
\]
and hence, by Lemma \ref{lem:ARin},
\[
 \operatorname{ar}_J(I_k)=\operatorname{ar}_J(I) = \ar_J(f_1,\ldots,f_r).
\]
Since $K_k\subseteq I_k$, Lemma \ref{lem:ARquot} yields
\[
 \operatorname{ar}_{JM_k}(f_kM_k)\le \operatorname{ar}_J(I_k)= \ar_J(f_1,\ldots,f_r).
\]
Therefore
\[
 \max\left\{
 a_J(0_{M_k}:f_k),
 \operatorname{ar}_{JM_k}(f_kM_k)+1
 \right\}+1
 \le N^\sharp.
\]

Since $\varepsilon_k\in J^{N^\sharp}$, \cite[Proposition 3.4]{VDT}
applied to the module $M_k$ and the element $f_k$ gives, for every
$j\ge0$,
\[
 \beta_j^R(R/I_k)=\beta_j^R(R/I_{k-1})
\]
and
\[
 \mu_R^j(R/I_k)=\mu_R^j(R/I_{k-1}).
\]
Chaining these equalities for $k=r,r-1,\ldots,1$ gives
\[
 \beta_j^R(R/I)=\beta_j^R(R/I'),
 \qquad
 \mu_R^j(R/I)=\mu_R^j(R/I')
\]
for every $j\ge0$.
\end{proof}

\begin{corollary}\label{cor:homological-dimensions}
Under the hypotheses of Theorem \ref{thm:homological-app},
\[
 \operatorname{pd}_R(R/I)=\operatorname{pd}_R(R/I'),
 \qquad
 \operatorname{id}_R(R/I)=\operatorname{id}_R(R/I'),
\]
and
\[
 \operatorname{depth}_R(R/I)=\operatorname{depth}_R(R/I').
\]
Moreover the zeroth Bass numbers, equivalently the socle dimensions, agree:
\[
 \dim_k\operatorname{Soc}(R/I)
 =\dim_k\operatorname{Soc}(R/I').
\]
\end{corollary}

\begin{proof}
Projective dimension, injective dimension and depth are read from the
vanishing patterns of the Betti and Bass numbers.  The last equality is the
case $j=0$ of the Bass-number statement.
\end{proof}

\begin{corollary}[Index of reducibility in the finite-length case]\label{cor:index-red}
Assume in addition that $R/I$ (equivalently, under the perturbation above,
$R/I'$) has finite length.  Then
\[
 \operatorname{ir}(I)=\operatorname{ir}(I').
\]
\end{corollary}

\begin{proof}
For an $\mm$-primary ideal $Q$ in a Noetherian local ring,
\[
 \operatorname{ir}(Q)=\dim_k\operatorname{Soc}(R/Q);
\]
see, for example, \cite[Remark 2.2]{CQT}.  Apply
Corollary \ref{cor:homological-dimensions}.
\end{proof}

\section*{Acknowledgements}
This work is supported by Vietnam National Program for the Development of Mathematics 2021-2030 under grant number B2027-CTT-02.
The author would like to thank Prof. Pham Hung Quy for suggesting questions concerning explicit perturbation bounds.

\end{document}